\documentclass[12pt,english]{article}
\usepackage[latin9]{inputenc}
\usepackage{babel}
\usepackage{setspace}
\usepackage{amsmath, mathtools, amsthm, amssymb}
\usepackage{enumerate}
\usepackage{setspace}
\usepackage{commath}
\usepackage{mdwlist}
\usepackage{placeins}
\usepackage{hyperref}
\usepackage{bookmark}
\makeatletter

\usepackage{amsfonts}\usepackage{amsthm}\usepackage{upref}\usepackage{enumerate}\usepackage[parfill]{parskip}\usepackage{color}\usepackage{multimedia}\usepackage{graphics}\usepackage{epsfig}\usepackage{eurosym}
\newcommand{\comment}[1]{}

\newtheorem{theorem}{Theorem}[section]
\newtheorem{corollary}{Corollary}[section]
\newtheorem{proposition}{Proposition}[section]
\newtheorem{acknowledgement}[theorem]{Acknowledgement}
\newtheorem{notation}{Notation}[section]

\newcommand{\A}{\mathscr{A}_T}
\usepackage{color}
\usepackage{epsfig,psfrag,amsmath,amssymb,latexsym}
\usepackage{amscd}
\usepackage{amsfonts}
\usepackage{enumerate}
\usepackage{graphicx}
\usepackage{hyperref}
\usepackage{bbm}
\usepackage{mathrsfs}
\usepackage{stmaryrd} 
\usepackage{yhmath} 
\usepackage{tikz}
\usetikzlibrary{arrows,decorations.markings}
\usetikzlibrary{positioning}
\usetikzlibrary{shapes.multipart,positioning,fit}
\usetikzlibrary{calc,through}
\usepackage{mdframed}

\newtheorem{remark}{Remark}[section]

\newcommand{\PP}{\mathbb{P}}

\newcommand{\T}{\mathbb{T}}

\numberwithin{equation}{section}

\makeatother
\begin{document}
\title{Asymptotic error distribution for the Euler scheme with locally Lipschitz coefficients}
\author{Philip Protter
\thanks{Statistics Department, Columbia University, New York, NY 10027;  pep2117@columbia.edu; supported in part by
NSF grant DMS-1612758}
\hspace{3mm}Lisha Qiu%
\thanks{Statistics Department, Columbia University, New York, NY 10027;  lq2141@columbia.edu}
\hspace{3mm}and Jaime San Martin%
\thanks{CMM-DIM, UMI-CNRS 2807, Universidad de Chile; jsanmart@dim.uchile.cl; supported in part by BASAL project PBF-03}
 Stochastic Differential Equations, Locally Lipschitz Coefficients{\bf Secondary} 60J60, 60J65
 }
\date{\today }

\maketitle
\begin{abstract}
In traditional work on numerical schemes for solving stochastic differential equations (SDEs), it is usually assumed that 
the coefficients are globally Lipschitz. This assumption has been used to establish a powerful analysis of the numerical approximations of the solutions of stochastic differential equations. 
In practice, however, the globally Lipschitz assumption on the coefficients is on occasion too stringent a requirement to meet. Some Brownian motion driven SDEs used in applications have coefficients that are Lipschitz only  on compact sets. Reflecting the importance of the locally Lipschitz case, it has been well studied in recent years, yet some simple to state, fundamental results remain unproved. We attempt to fill these gaps in this paper, establishing both a rate of convergence, but also we find the asymptotic normalized error process of the error process arising from a sequence of approximations. The result is analogous to the original result of this type, established in~\cite{KT1991-2} back in 1991. This result was improved in 1998 in~\cite{JJ1998}, and recently(2009) it was partially extended in~\cite{AN}. As we indicate, the results in our paper provide the basis of a statistical analysis of the error; in this spirit we give conditions for a finite variance. 
\end{abstract} 
\section{Introduction}
We investigage the numerical solution of a one-dimensional stochastic differential equation (SDE) of the form
\begin{align} \label{Sde1}
dX_t=\mu(X_t)dt+\sigma(X_t)dW_t, \  0\leq t \leq T,\  X_0=x_0 \in \mathbb R.
\end{align}
Here $X_t \in \mathbb R$ for each $t$, $\mu,\sigma: \mathbb R \to \mathbb R$ are coefficient functions, and $W$ is a one dimensional standard Brownian motion.  We assume the initial value $x_0\in \mathbb R$ is non-random. For background information about SDEs, we refer to 
Chapter 5 of Protter \cite{PP1990}, Chapter 9 of Revuz and Yor \cite{RD2013} and Chapter 5 of Karatzas and Shreve \cite{KI2012}. 

In applications, one would often like to solve (\ref{Sde1}) numerically, as an explicit solution is usually not obtainable. This is often done in low dimensions using PDE methods that require heavy computational complexity. Hence, in practice, it is advisable to solve (\ref{Sde1}) with the simple Euler scheme. (See the survey paper of Talay \cite{TD1984} for a discussion of this issue). Our primary objective is to study uniform convergence in probability and weak convergence of the normalized error process for the Euler scheme under locally Lipschitz and no finite explosion assumptions on (\ref{Sde1}). Note that the result on uniform convergence in probability in this paper is not restricted to the Euler scheme, but applicable to all numerical schemes satisfying some mild assumptions.

The use of the Euler scheme to solve Brownian motion driven SDEs is already well studied. A number of treatments impose conditions on $\mu$ and $\sigma$ in (\ref{Sde1}), and in particular a globally Lipschitz condition and/or a linear growth condition is imposed. We list some of the works here. For the rate of convergence of the expectation of functionals, see Talay and Tubaro \cite{TD1990};
for the rate of convergence of the distribution function, see Bally and Talay  \cite{BV1996-1}; for the rate of 
convergence of the density, see Bally and Talay \cite{BV1996-2}; for error analysis, see
Bally and Talay \cite{BV1995}; for an Euler scheme when one has irregular coefficients and Hölder continuous coefficients see Yan \cite{YL2002}, and in this regard see also Bass-Pardoux~\cite{Bass}; for complete reviews, see Talay \cite{TD1990-2} and Kloeden-Platen \cite{KP1999}. In two interesting recent papers M. Bossy et al~\cite{Bossy}, \cite{Bossy2008} have studied a modified (symmetrized) Euler scheme to handle solutions of the Cox-Ingersoll-Ross type (CIR), but for equations where the diffusive coefficient is of the form $\vert x\vert^\alpha$ for $\frac{1}{2}\leq \alpha<1$, which are of course locally Lipschitz.

There is also some work on numerical schemes for solving SDEs not tied to Brownian motion, but rather driven by semimartingales with jumps. The case of SDEs driven by Brownian motion and Lebesgue measure  can be found in
Kurtz and Protter~\cite{KT1991-2} where the convergence in distribution of the normalized Euler scheme
is first studied. $L^p$ estimates of the Euler scheme error were given by Kohatsu-Higa and
Protter \cite{KH1994}.  Protter and Talay \cite{PP1997} also studied the Euler scheme for SDEs driven
by Lévy processes. Jacod and Protter~\cite{JJ1998} obtained a (to date) definitive result about the
asymptotic error distributions for the Euler scheme solving SDEs driven by a vector of
semimartingales. 
{
 More recent work has focused on numerical schemes to solve SDEs under relaxed conditions on the coefficients, to wit the locally Lipschitz condition replaces the customary Lipschitz condition. Under the locally Lipschitz hypothesis, the Euler scheme may diverge in the strong sense of convergence, such as $L^p$.  The $L^p$ convergence, or more correctly the lack of it, is studied in Hutzenthaler, Jentzen and Kloeden~\cite{MH2011}.}  To obtain convergence results for the Euler scheme under the locally Lipschitz condition, additional assumptions are assumed in existing work. 
Examples of attempts are assuming the existence of a Lyapunov function, or a one sided Lipschitz condition and finite moments of the true solution and the numerical solution (see~\cite{GI},\cite{HD2002},\cite{MG2002} ). 

Convergence in probability for Euler-type schemes in general still holds, see Hutzenthaler and Jentzen~\cite{MH2015} and the citations therein.
Under the condition that $\mu,\sigma$ are continuously differentiable ($\mathcal{C}^1$) and grow at most linearly, Kurtz and Protter \cite{KT1991-2} obtained the limit distribution for the asymptotic normalized error process for the Euler scheme.   
Neuenkirch and Zähle~\cite{AN} generalized the result of Kurtz and Protter by assuming the solution never leaves an open set in finite time and that the coefficients are $\mathcal{C}^1$.

In this paper, we study the limit distribution for the asymptotic normalized error process with only a locally Lipschitz assumption plus no finite time explosions, and $\sigma$ in (\ref{Sde1}) being bounded away from $0$. 
By relaxing the $\mathcal{C}^1$ and linear growth hypotheses to the assumption that the coefficient need only be locally Lipschitz, we are able to deal with coefficients that may have super linear growth, and their derivatives may have poor smoothness properties, or may not even exist.

Some locally Lipschitz coefficients lead to well defined stochastic differential equations, but only because the solution remains always positive. This is the case for example with the CIR type processes. The Euler scheme approximations, however, need not be defined, since for example we might be taking the square root of a negative quantity at some steps. For these situations, we can use a nice trick due to Bossy at al~\cite{Bossy2008,Bossy} where the Euler scheme is replaced by what is known as a symmetrized Euler scheme. This keeps the approximations positive, too. Our results apply for these schemes as well, since they are ''local'', which is our rubric for the types of schemes we utilize in this paper.

This paper is organized as follows. In Section \ref{Section2}, we prove that if a numerical scheme converges uniformly in probability on any compact time interval with a certain rate under the globally Lipschitz condition, then the same result holds when the globally Lipschitz condition is replaced with a locally Lipschitz condition and a no finite time explosion condition. 
The Euler and Milstein schemes are studied as examples. From Section 3 on, we focus on the Euler scheme. We prove that the sequence of the error process for the Euler scheme normalized by $\sqrt{n}$ is relatively compact. Furthermore, by proving uniqueness of the limit process, we show the normalized error process converges in law. The limit error process is also provided as a solution to an SDE. This is not surprising, given the results of~\cite{KT1991-2}. Section 4 turns to a study on the the second moment of the weak limit process and its running maximum. In the last section, we give an upper bound for the rate of weak convergence for the approximating expectation of functionals for the Euler scheme. 

\bigskip

\begin{acknowledgement}
We wish to thank Jean Jacod and Denis Talay for helpful discussions during our work on this paper. The results in this paper constitute part of the PhD thesis of Lisha Qiu,  in the Statistics Department of Columbia University. The third author is very grateful for the hospitality of the Statistical Department of Columbia University.
\end{acknowledgement}


\section{Convergence in Probability}\label{Section2}
We start with some notation. For a discretization of the time interval $[0,T]$ with discretization size $\frac{T}{n}$, let $n(t)=[\frac{nt}{T}]\frac{T}{n}$, the nearest left time grid point for $t$. For all
$g :[0,T]\rightarrow \mathbb R$, define 
\begin{align} \label{delta}
\Delta g_t^{(n)}=g(t)-g({n(t)}).
\end{align}  
The continuous Euler scheme for solving SDE (\ref{Sde1}) is defined by
\begin{align} 
X_t^{E,n}&=X_{n(t)}^{E,n}+\sigma(X_{n(t)}^{E,n})\Delta W_t^{(n)}+\mu(X_{n(t)}^{E,n})\Delta t^{(n)}, \  X^{E,n}_0=X_0, 
\end{align}
and the continuous Milstein scheme is defined by
\begin{align*} 
X_t^{M,n}&=X_{n(t)}^{M,n}+\sigma(X_{n(t)}^n)\Delta W_t^{(n)}+\mu(X_{n(t)}^n)\Delta t^{(n)}+
\frac{1}{2}\sigma(X_{n(t)}^{M,n})\sigma^\prime(X_{n(t)}^{M,n})[(\Delta W_t^{(n)})^2-\Delta t^{(n)}]\\&+
\frac{1}{2} \mu(X_{n(t)}^{M,n})\mu^\prime(X_{n(t)}^{M,n}) (\Delta t^{(n)})^2, \quad X^{M,n}_0=X_0.
\end{align*}

Under the globally Lipschitz condition, most of the proposed numerical schemes including the Euler and Milstein schemes have been proved to converge uniformly in probability at a finite time point. Fortunately, the same result can be extended to the locally Lipschitz case if one also adds a no finite time explosion condition. 
To prove this, we need a localization technique. Let us start with some notation. 

\bigskip

\begin{notation}
Given a process $Z$, we denote by $\T^m(Z)=\inf\{ t\ge 0: |Z_t|>m\}$. 
Also, we denote by $Z^\T$ the stopped process.
\end{notation}
In what follows, we denote by $X=X(x_0,\mu,\sigma,W)$ the unique
solution of the SDE (\ref{Sde1}),
where the coefficients $\mu,\ \sigma$ are assumed regular enough to have a unique strong solution (for example locally Lipschitz).
For every $m \ge 1$  consider $\mu^{(m)}$ a continuous modification of $\mu$ such that $\mu(x)=\mu^{(m)}(x)$ for $|x| \le m$,
$\mu^{(m)}(x)=\mu(m+1)$ for $x \ge m+1$ and $\mu^{(m)}(x)=\mu(-m-1)$ for $x \le -m-1$.
In case the numerical procedure assumes that $\mu$ is ${\cal C}^k$ (or Lipschitz) we interpolate $\mu^{(m)}$ 
on $(-m-1,-m)\cup (m,m+1)$ in such a way that $\mu^{(m)}$ is also ${\cal C}^k$ (respectively Lipschitz). 
Similarly, we denote by $\sigma^{(m)}$ a modification of $\sigma$.
Given a numerical procedure $\phi$, we denote by $(X^{\phi,n})_n=(X^{\phi,n}(x_0,\mu,\sigma,W))_n$ the associated sequence of approximations. We remove the dependence on $\phi$ in $X^{\phi,n}$ when there is no possible confusion.
Note that we use the same Brownian motion for every $n$.
This numerical procedure is assumed {\bf local} in the following sense. Assume that $\mu=\tilde \mu, \sigma=\tilde \sigma$
on the interval $[-m,m]$, where $|x_0|<m$. 
Then for all $n$  and for $\T=\T^m(X^{\phi,n}(x_0,\mu,\sigma,W))$ it holds 
$$
(X^{\phi,n}(x_0,\mu,\sigma,W))^\T=(X^{\phi,n}(x_0,\tilde \mu,\tilde \sigma,W))^\T \,
$$
almost surely. In particular $\T^m(X^{\phi,n}(x_0,\mu,\sigma,W))=\T^m(X^{\phi,n}(x_0,\tilde \mu,\tilde \sigma,W))$ a.s..
This hypothesis is satisfied, for example, by the Euler and Milstein schemes. 
On the other hand, if $(\mu,\sigma)$ and $(\tilde \mu,\tilde \sigma)$ are regular, the associated solutions
satisfy
$$
(X(x_0,\mu,\sigma,W))^{\T}=(X(x_0,\tilde \mu,\tilde \sigma,W))^{\T}\
$$
almost surely for $\T=\T^m(X(x_0,\mu,\sigma,W))$.
Again, we have $\T^m(X(x_0,\mu,\sigma,W))=\T^m(X(x_0,\tilde \mu,\tilde \sigma,W))$ a.s..
\bigskip
Now, we present Theorem \ref{th.CVP}.

\begin{theorem}\label{th.CVP}
Assume that a numerical scheme $\phi$ is well defined and local. Let $X^{n}_t$ be the numerical solution using $\phi$ for the SDE (\ref{Sde1}) on 
$[0,T]$. If $X^{n}_t$ converges in probability uniformly on $[0,T]$, with order $\alpha>0$, that is $\forall \ C>0$
\begin{align} \label{cvp-eq1}
\mathbb P(n^\alpha  \underset{0\leq t \leq T} {\sup} \abs{X^{n}_t-X_t}>C )\rightarrow 0, \quad as \ n \rightarrow +\infty, \ 
\end{align}
given $\mu, \sigma$ are globally Lipschitz, then (\ref{cvp-eq1}) also holds when the globally Lipschitz condition is replaced with a locally Lipschitz condition and a no finite time explosion condition. 
\end{theorem}
\begin{proof}


In what follows, to avoid overly burdensome notation, we denote by 
$$
\begin{array}{l}
X=X(x_0,\mu,\sigma,W),\, X^n=X^{n}(x_0,\mu,\sigma,W),\\
\\
Y^{(m)}=X(x_0,\mu^{(m)},\sigma^{(m)},W),\, Y^{n,(m)}=X^{n}(x_0,\mu^{(m)},\sigma^{(m)},W)\,.
\end{array}
$$ 

Now, define 
$$
\mathscr{X}_n=\left\{\omega:\, n^\alpha \sup\limits_{0\le t\le T} |X^n_t-X_t|\le C\right\},\,
\mathscr{Y}_{n,(m)}=\left\{\omega:\, n^\alpha \sup\limits_{0\le t\le T} |Y^{n,(m)}_t-Y^{(m)}_t|\le C\right\}\,.
$$
We also consider $\T=\T^{m}(X), {\cal S}=\T^{m-1}(X)$. It is clear that $X^{\cal S}=(Y^{(m)})^{\cal S}$
(actually they are equal up to time $\T$). Since the numerical procedure is local, also 
${\cal U}=\T^{m}(X^n)=\T^{m}(Y^{n,(m)})$ and
$$
(X^n)^{\cal U}=(Y^{n,(m)})^{\cal U}\,.
$$

Consider $m$ large enough such that $|x_0|<m-1$ and $n$ large enough such that $C/n^\alpha<1$. 
For these values of $n,m$, we show that a.s.
$$
\mathscr{Y}_{n,(m)} \bigcap \{{\cal S}>T\}\subset \mathscr{X}_{n}\,.
$$
Indeed, on the set $\{{\cal S}>T\}$ the two processes $X, Y^{(m)}$ agree on $[0,T]$ a.s.. In particular, 
we have that $\sup\limits_{0\le t\le T} |Y^{(m)}_t|=\sup\limits_{0\le t\le T} |X_t|\le m-1$ a.s. 
On the other hand on the set $\mathscr{Y}_{n,(m)} \bigcap \{{\cal S}>T\}$
we have a.s.
$$
\sup\limits_{0\le t\le T} |Y^{n,(m)}_t|\le m-1+\frac{C}{n^\alpha}< m\,.
$$
That is $\T^{m}(Y^{n,(m)})>T$ a.s.. Since $\phi$ is local, we deduce that on $[0,T]$ the processes
$X^n$ and $Y^{n,(m)}$ agree a.s. and therefore on $\mathscr{Y}_{n,(m)} \bigcap \{{\cal S}>T\}$
$$
\sup\limits_{0\le t\le T} |X^n-X_t|=\sup\limits_{0\le t\le T} |Y^{n,(m)}-Y^{(m)}_t|\le \frac{C}{n^\alpha}
$$
holds also a.s., proving the desired inclusion. This shows that the inequality
$$
\PP\left(n^\alpha \sup\limits_{0\le t\le T} |X^n_t-X_t|> C\right)\le
\PP\left(n^\alpha \sup\limits_{0\le t\le T} |Y^{n,(m)}_t-Y^{(m)}_t|> C\right) + \PP(\T^{m-1}(X)\le T)\,.
$$
Now, given $\epsilon>0$ choose $m$ large such that $\PP(\T^{m-1}(X)\le T)\le \epsilon/2$. For that $m$, 
according to the hypothesis of the Theorem, there exists $n_0=n_0(m,\epsilon)$, such that for all $n\ge n_0$
$$
\PP\left(n^\alpha \sup \limits_{0\le t\le T} |Y^{n,(m)}_t-Y^{(m)}_t|> C\right)\le \frac{\epsilon}{2}\,,
$$
giving the result.
\end{proof}


Taking $\alpha=0$ in Theorem \ref{th.CVP}, we immediately see that if a numerical scheme converges in
probability uniformly on compact time intervals for solving SDEs with the globally Lipschitz coefficients, then the same result also holds under the locally Lipschitz condition and no finite time explosion condition.   

We illustrate the application of Theorem \ref{th.CVP} using the two most widely used numerical schemes, the Euler scheme and the Milstein scheme.  Under the globally Lipschitz condition, it is well known that the continuous Euler and Milstein schemes converge in probability uniformly on compact time intervals with any order between $[0,\frac{1}{2})$ and $[0,1)$ respectively. Interested readers can refer to  \cite{TD1983} and \cite{YL2002} for details. Then applying Theorem \ref{th.CVP} leads to the following corollary.
\begin{corollary}\label{cor1}
 Consider SDE (\ref{Sde1}), assume $\mu, \sigma$ are locally Lipschitz and the solution has no finite time explosion and the continuous Euler scheme $X^{E,n}$
and the continuous Milstein scheme $X^{M,n}$ are well defined for solving (\ref{Sde1}).  Then $X^{E,n}$ 
and $X^{M,n}$, converge in probability uniformly to $X$ on $[0,T]$. 
Moreover $\forall \ \gamma \in (0,\frac{1}{2}]$, we have
\begin{align*} 
&\mathbb P(n^{\frac{1}{2}-\gamma}\underset{0\leq t \leq T} {\sup} \abs{X^{E,n}-X_t}>C )\rightarrow 0, \quad as\ n \rightarrow +\infty, \\
&\mathbb P(n^{1-\gamma} \underset{0\leq t \leq T} {\sup} \abs{X^{M,n}-X_t}>C )\rightarrow 0, \quad as\ n \rightarrow +\infty.
\end{align*}
\end{corollary}


\section{Asymptotic Error Distribution for the Euler Scheme}

In this section, we are going to prove that the asymptotic normalized error process from the Euler scheme converges in distribution with rate $\sqrt{n}$, under the locally Lipschitz and no finite time explosion assumption.
In the previous section, the localization technique used in the proof for Theorem \ref{th.CVP} transfers the locally Lipschitz case into the globally Lipschitz case. The localization technique will be used in this section as well, and we present Proposition \ref{lemma3-1} to make future proofs concise when applying this technique.\\

\begin{proposition}\label{lemma3-1} 
Consider the SDE (\ref{Sde1}),  assume that $\mu$ and $\sigma$ are locally Lipschitz and the solution $X$ has no finite time explosion. For 
$m>\abs{x_0}$, define $Y^{(m)}$ as in the proof of Theorem \ref{th.CVP}.
Let $X^n$, $Y^{n,(m)}$ be the numerical solutions from a numerical scheme $\phi$. Assume $\phi$ is well defined, local and converges uniformly in probability on compact time interval. 
 Then $\forall \ 0<T<\infty$, 
\begin{align*}
\underset{m\to \infty}{\lim} \underset{n}{\sup} \ \mathbb P \Big(X\neq Y^{(m)}  \text{ or }X^n\neq Y^{n,(m)} \Big)=0, \text{ on $[0,T].$}
\end{align*}
\end{proposition}
\begin{proof}
Since the Euler scheme is local, 
\begin{align}
 \T^m(X)=\T^m(X^{(m)}), \  \T^m(X^n)=\T^m(X^{n,(m)}).
\end{align} 
Thus, we have on $[0,T]$
\begin{align}\label{Lemma1-eq1}
 \mathbb P \Big(X \neq Y^{(m)}  \text{ or }X^n\neq Y^{n, (m)} \Big)=
 \mathbb P \Big( \T^{m}(X)<T  \text{ or }   \T^{m}(X^n)<T  \Big).
\end{align}
Since $X$ has no finite time explosion, 
$\forall \epsilon>0$, there exists $ m_1=m_1(\epsilon)$ large enough so that  
\begin{align*} 
\mathbb P( \T^{m_1}(X)\le T)< \frac{\epsilon}{3}.
\end{align*}
By the uniform convergence in probability of $X^n$ on $[0,T]$, there exists $n^\prime=n'(\epsilon)$ 
such that $\forall n>n^{\prime}$, we have
\begin{align*} 
\mathbb P(\underset{0\leq s \leq T} {\sup} \abs{X^n_s-X_s} \geq 1)<\frac{\epsilon}{3}
\end{align*}
and 
 \begin{align*} 
\mathbb P(\T^{m_1+1}(X^n) \leq T)\leq \mathbb P(\T^{m_1}(X) \leq T)+
\mathbb P(\underset{0\leq s \leq T} {\sup} \abs{X^n_s-X_s} \geq 1)< \frac{2}{3}\epsilon.
\end{align*}
Hence when $n>n^{\prime}$, 
 \begin{align*}  
\mathbb{P}(\T^{m_1+1}(X^n)\le T  \text{ or } \T^{m_1+1}(X)\le T)  \leq \mathbb P(\T^{m_1}(X) \leq T)+
\mathbb P( \T^{m_1+1}(X^n) \leq T) <\epsilon.
\end{align*}
Now for $n\leq n^\prime$, take $Q_t^\star=\max \Big \{ { X^*_t, X^{1,*}_t, X^{2,*}_t \dots X^{n^\prime,*}_t} \Big \}$, 
where $X^*$ indicates the running maximum of the absolute process. As the true solution $X$ and the Euler 
scheme numerical solutions have no explosion on  $[0,T]$, and $n^\prime $ is finite, the process 
$Q_t^\star$ has no explosion either. We can always find $m=m(n')>m_1+1$ large enough so that for the hitting time 
of $Q_t^\star$,  $\mathbb P(\T^{m}(Q_t^\star)  \leq T)\le \epsilon$. 
Together with (\ref{Lemma1-eq1}), the proof concludes.
\end{proof}
\begin{remark}\label{r1}
Proposition \ref{lemma3-1} also holds for the multidimensional case with the same techniques used in the proof. 
\end{remark}

\bigskip

\begin{remark}\label{r2}
Kurtz and Protter \cite{KT1991-2} obtained the weak limit for the sequence of normalized error process for the Euler scheme under the condition that $\mu, \sigma$ are $\mathcal{C}^1$ and of at most linear growth. Proposition \ref{lemma3-1} implies that the condition can be replaced with  $\mu, \sigma$ are $\mathcal{C}^1$ and at most linear growth condition on any compact set plus the solution has no finite time explosion.  This generalization can be found in Neuenkirch and Zähle \cite{AN}.  
\end{remark}

\bigskip
\begin{remark}
In Yan \cite{YL2005}, the $\mathcal{C}^2$ and at most growth hypotheses used for obtaining  for the weak limit of the sequence of normalized error process for the Mistein scheme, can be replaced with $\mathcal{C}^2$ and at most linear growth conditions on any compact set, plus that the solution has no finite time explosion. 
\end{remark}

\bigskip
Now we turn to prove a weak convergence result for the normalized error of the Euler scheme with the locally Lipschitz assumption.
Define $Z^n$ as follows
 \begin{align*}
Z_t^{n11}&= \int_0^t  \sqrt{n}\Delta s^{(n)}ds, \
Z_t^{n12}= \int_0^t  \sqrt{n}\Delta s^{(n)}dW_s,\\
Z_t^{n21}&= \int_0^t  \sqrt{n}\Delta W_s^{(n)}ds,\ 
Z_t^{n22}= \int_0^t  \sqrt{n}\Delta W_s^{(n)}dW_s.
\end{align*} 
Our goal is to prove convergence in distribution for the asymptotic error process from the Euler scheme at the rate $\sqrt{n}$, which requires $Z^n$ to converge in distribution. \\

\begin{proposition}\label{propZ}
The sequence $Z^n$ is tight and converges in distribution to $Z$ under the uniform topology on compact time set, where $Z$ is independent of $W$ and $Z^{1,1}=Z^{1,2}=Z^{2,1}=0$, $\sqrt{2}Z^{2,2}$ is a standard Brownian motion.
\end{proposition}
{
Proposition~\ref{propZ}} is implied by Theorem 5.1 in Jacod and Protter \cite{JJ1998}.\\


\begin{proposition}\label{lemma3-2}
Consider SDE (\ref{Sde1}), and assume that $\mu(x), \sigma(x)$ are both Lipschitz and bounded. Let $X^n$ be the numerical solution 
to (\ref{Sde1}) on $[0,T]$ from the continuous Euler scheme with step size $\frac{T}{n}$. Then the sequence of normalized error processes 
$U_n=\sqrt{n}(X^n-X)$ is relatively compact.
\end{proposition}
\begin{proof}
It has been proved that $Z^n$ are good sequences (see \cite{KT1991-1} for the definition of a good sequence).  From 
Proposition \ref{propZ}, $Z^n\Rightarrow Z$, where $\Rightarrow$ denotes convergence in distribution under the uniform topology on a compact time set.  The limit process $Z$ is independent of $W$ and $Z^{1,1}=Z^{1,2}=Z^{2,1}=0$, $Z^{2,2}$ is mean zero Brownian motion with $\mathbb{E}[(Z^{2,2}_t)^2]=\frac{t}{2}$.
By Corollary \ref{cor1}, we also have 
$(X^n,Z^n)\Rightarrow (X,Z)$. 
By the definition of a continuous Euler scheme, $X^n$ can also be represented as 
\begin{align*} 
X^n_t=\int_0^t \mu(X^n_{n(s)}) ds +\int_0^t \sigma(X^n_{n(s)}) dW_s.
\end{align*}
Then
\begin{align*} 
U^n_t&=\sqrt{n}(X^n_t-X_t)\\
&=\int_0^t \sqrt{n} \{ \mu(X^n_{n(s)})-\mu(X_s)\}ds+\int_0^t  \sqrt{n} \{ \sigma(X^n_{n(s)})-\sigma(X_s)\}dW_s .
\end{align*}
For $x\neq y$, define functions $g,h: \mathbb R^2\to \mathbb R$ as
\begin{align*} 
g(x,y)=\frac{\mu(x)-\mu(y)}{x-y},\quad h(x,y)=\frac{\sigma(x)-\sigma(y)}{x-y}.
\end{align*}
Since $\mu,\sigma$ are Lipschitz, $g(x,y)$ and $h(x,y)$ are bounded. 
Now we separate the error process into two terms $U^n=U^{1,n}+U^{2,n}$, where
\begin{align*} 
U^{1,n}_t=&\int_0^t \sqrt{n} \{ \mu(X^n_{n(s)})-\mu(X_s)\}ds\\
&= \int_0^t \sqrt{n}\{ \mu(X^n_s)-\mu(X_s)\}ds -\int_0^t \sqrt{n}\{ \mu(X^n_s)-\mu(X^n_{n(s)})\}ds\\
&=\int_0^t  g(X^n_s,X_s) U^{n}_t ds-\int_0^t   \frac{ \mu(X^n_s)-\mu(X^n_{n(s)})} {X_s-X^n_{n(s)}}(X^n_s-X^n_{n(s)}) \sqrt{n} ds.
\end{align*}
Note that $X^n_s-X^n_{n(s)}=\mu(X^n_{n(s)}) \Delta s^{(n)}+\sigma(X^n_{n(s)}) \Delta W_s^{(n)}$. Then
\begin{align*} 
U^{1,n}_t&=\int_0^t  g(X^n_{n(s)},X_s) U^{n}_t- g(X^n_s,X^n_{n(s)}) \big \{ \mu(X^n_{n(s)}) \sqrt{n}\Delta s^{(n)}+\sigma(X^n_{n(s)}) \sqrt{n}\Delta W_s^{(n)}\big\}ds.
\end{align*}
Similarly, 
\begin{align*} 
U^{2,n}_t&=\int_0^t  h(X^n_{n(s)},X_s) U^{n}_t -h(X^n_s,X^n_{n(s)}) \big \{ \mu(X^n_{n(s)}) \sqrt{n}\Delta s^{(n)}+\sigma(X^n_{n(s)}) \sqrt{n}\Delta W_s^{(n)}\big\}dW_s.
\end{align*}
For notational convenience, define $\tilde{f}^n$ as
\begin{align*} 
\tilde{f}^n=\big[ g(X^{n}_{s},X_s),g(X^{n}_s,X^{n}_{{n}(s)}),
h(X^{n}_{s},X_s),h(X^{n}_s,X^{n}_{{n}(s)})\big].
\end{align*} 
If $\mu,\sigma$ are also assumed to be continuously differentiable, as in Kurtz and Protter~\cite{KT1991-2}, then $\tilde{f}^n$ converges weakly uniformly to $[\mu^\prime(X),\mu^\prime(X),\sigma^\prime(X),\sigma^\prime(X)]$ on $[0,T]$. By results on weak convergence of stochastic integrals in Kurtz and Protter~\cite{KT1991-1}, $U^n$ converges weakly uniformly on $[0,T]$ as well. 

However, here $\sigma,\  \mu$ are only assumed to be Lipschitz and bounded, hence their derivatives might not be continuous or not even exist.  This would cause $\tilde{f}^n$ to fail to converge weakly.  Fortunately, by the boundedness of $\tilde{f}^n$, applying weak convergence techniques in~\cite{KT1991-1} would give relative compactness of $U^n$ under the uniform topology, which is shown in the following steps. 
   
By Prokhorov's Theorem which states that tightness is equivalent to relative compactness in our case, $\tilde{f}^n$ is also relatively compact. Then for every subsequence of  $\tilde{f}^n$, there exists a further subsubsequence $n_k$ such that $\tilde{f}^{n_k}$ converges weakly uniformly on $[0,T]$.
It is also known that $(X^n_{n(.)},X^n,\sqrt{n} Z^n)\Rightarrow (X,X,Z)$, and the sequence is a good sequence (see \cite{KT1991-2} for details). Then, we can assume on $[0,T]$,
\begin{align*} 
&\big[
f^{n_k,1}, f^{n_k,2},f^{n_k,3},f^{n_k,4},X^{n_k}_{n_k(.)},X^{n_k},Z^{n_k,1}, Z^{n_k,2},Z^{n_k,3},Z^{n_k,4}
\big] \\
&\Rightarrow [G,\tilde{G},H,\tilde{H}, X,X,0,0,0,\frac{\sqrt{2}}{2}B].
\end{align*} 
Since $Z^n$ is a good sequence and $\mu,\sigma$ are bounded, then, by proof of Theorem 3.5 in Kurtz and Protter~\cite{KT1991-2}, $U^{n_k}\Rightarrow R$ on $[0,T]$, where 
\begin{align} \label{Sde4}
R_t=\int _0^t G_t R_tds+\int _0^t H_t R_tdW_s +\frac{\sqrt{2}}{2} \int _0^t \sigma(X_t) \tilde{H}_t dB_s.
\end{align} 
Thus every subsequence of $U^n=\sqrt{n}(X^n-X)$ has a subsubsequence that converges weakly uniformly on $[0,T]$, implying that $U^n$ is relatively compact.
\end{proof}

\bigskip

\begin{remark}\label{r4}
Our next theorem is similar to results in~\cite{KT1991-2,AN} but with two important differences: We do not assume the coefficients are $\mathcal{C}^1$, but only that they are locally Lipschitz; We do not assume a linear growth condition, but rather assume only locally Lipschitz combined with no finite explosions in finite time. As a simple example, this allows for the consideration of coefficients of the form $\sigma(x)=x^\gamma$, with $\gamma >1$. In Economics, such coefficients are known as CEV (= Constant Elasticity of Variance). Usually $\gamma$ is assumed to be less than or equal to one, but here we lay the groundwork to consider $\gamma >1$ on a practical level. 
\end{remark}

\bigskip


\begin{theorem}\label{th.3-2}
Consider the SDE (\ref{Sde1}), assume that $\mu, \sigma$ are locally Lipschitz and that the solution $X$ has no finite time explosion. Further assume that
 $\sigma(x)$ is non-negative and bounded from below by some $d \in \mathbb R^+$ on any compact set. Let $\mu^\prime(x), \sigma^\prime(x)$ equal the derivatives of 
 $\mu,\sigma$ at $x$ when the derivatives exist; and that they equal $0$, when the derivatives at a point $x$ do not exist. 
 
Let $X^n$ be numerical solution from the continuous Euler scheme, and $U_n=\sqrt{n}(X^n-X)$ be the normalized error process. Then for all $ 0<T<\infty$, $U_n$ converges weakly uniformly on $[0,T]$ to $U$, where $U$ satisfies
\begin{align} \label{Sde4}
U_t=\int _0^t \mu^\prime(X_s) U_sds+\int _0^t \sigma^\prime(X_s) U_sdW_s +\frac{\sqrt{2}}{2} \int _0^t \sigma(X_s) \sigma^\prime(X_s) dB_s, \ U_0=0,
\end{align} 
where $B$ is a standard Brownian motion and is independent of $W$.
\end{theorem}
\begin{proof}
Use Proposition \ref{lemma3-1} and apply the localization technique, we can assume $\mu,\sigma$ 
are bounded and globally Lipschitz and there exists $d>0$, such that for all $x \in \mathbb R$ we have 
$\abs{\sigma(x)}>d$ without loss of generality.  In what
follows we denote by $K$ a constant that bounds $|\mu|,|\sigma|$ and the Lipschitz constants of $\mu,\sigma$.

Define $g(x,y),h(x,y)$ as in Proposition \ref{lemma3-2}.  By Proposition \ref{propZ}
\begin{align}
(Z^{n11},Z^{n12},Z^{n21},Z^{n22})\Rightarrow Z=(0,0,0,\frac{\sqrt{2}}{2}B),  \ \text{on $[0,T]$.}
\end{align}
$B$ is a standard Brownian motion and is independent of $W$. 
Proposition \ref{lemma3-2} shows $U^n_t$ is relatively compact. Thus for
 any subsequence $n^\prime$, there exists a subsubsequence 
$n_k^{\prime}$ of $n^{\prime}$ and a process $R$ in $\mathcal{C}[0,T]$, such that $U^{n_k^{\prime}}_t \Rightarrow R$. SDE (\ref{Sde4}) has unique weak solution as it satisfies the Engelbert-Schmidt conditions, (see \cite{EH1985} for details). 
To prove $U^n \Rightarrow U$, it is sufficient to prove that $R$ is a weak solution to SDE (\ref{Sde4}).

Because $(U^{n_k^{\prime}},X,W,Z^n) \Rightarrow (R,X,W,Z)$, by the almost sure representation theorem (Theorem 1.10.4 on
 page 59 of van der Vaart and Wellner \cite{VDV1996}), there exists a probability space $(\bar{\Omega},\bar{\mathcal{F}},\bar{P})$ and 
 a sequence of processes $\tilde{Y}_k$ and $Y$, with 
 $\mathcal {L} (Y^k)=\mathcal{L}(U^{n_k^{\prime}},X,W,Z^n)$ for all $k\geq 1$, 
such that $\mathcal {L} (Y)=\mathcal{L}(R,X,W,Z)$, and $Y^k \overset{a.s.}{\to}Y$ uniformly on $[0,T]$. 
If we could prove that the first element of $Y$ is a weak solution to SDE (\ref{Sde4}), it follows immediately that $R$ is also a weak solution to (\ref{Sde4}). 

Thus, without loss of generality, we assume $(U^{n},X,W,Z^n)\overset{a.s.}{\to} (R,X,W,Z)$  as 
$n\rightarrow \infty$ and we try to prove $R$ is a weak solution to (\ref{Sde4}). In particular, for fixed $T>0$ we remove $\A$
a set of probability  $\PP(\A)=0$, such that for all $\omega\in \A^c$ and uniform 
in $[0,T]$, we have the convergence
$$
(U^{n},X,W,Z^n)\to (R,X,W,Z).
$$
We first present one known result for the continuous Euler scheme under the condition that $\mu,\sigma$ 
are globally Lipschitz, stated here as (\ref{bdd1}). The proof of (\ref{bdd1}) can be found in 
 Kloeden \cite{KP1999}, proof of Theorem 10.2.2.
\begin{align}\label{bdd1}
\underset{n}{\sup} \ \mathbb E \underset{0<s\leq T}{\sup} \abs{U^{n}_s}^2< \infty.
\end{align} 
Since $U^n\overset{a.s.}{\to} R$ on $[0,T]$, by Fatou's lemma, we also have
\begin{align}\label{bd2}
 \ \mathbb E \underset{0<s\leq T}{\sup} \abs{R_s}^2< \infty.
\end{align} 
 From the definition of $U^n$, we have $U^n=\sqrt{n}(X^n-X)=U^{1,n}+U^{2,n}$, where $U^{1,n},U^{2,n}$ are the same as in proof of Lemma \ref{lemma3-2}.
Since the Lipschitz condition implies differentiability almost everywhere, we can find subset $A$ of $\mathbb{R}$ with Lebesgue measure 0 such that both $\mu$ and 
$\sigma$ are differentiable on $\mathbb{R}\cap A^c$. Define $I_1=I_{\{s: X_s \in  A^c \}}$ and $I_2=I_{\{s: X_s \in A\}}$. We analyze the following terms, $i=1,2$,
\begin{align*} 
G^{ni1}_t&=\int_0^t  I_i \ \{g(X^n_{n(s)},X_s) U^{n}_t-\mu^\prime(X_t) R_t \}ds, \\
G^{ni2}_t&=\int_0^t  I_i \ g(X^n_s,X^n_{n(s)})\mu(X^n_{n(s)}) \sqrt{n}\Delta s^{(n)} ds,\\
G^{ni3}_t&=\int_0^t  I_i \ g(X^n_s,X^n_{n(s)})\sigma(X^n_{n(s)}) \sqrt{n} \Delta W_s^{(n)} ds,\\
F^{ni1}_t&=\int_0^t   I_i \  \{h(X^n_{n(s)},X_s) U^{n}_t-\sigma^\prime(X_t) R_t \}dW_s, \\
F^{ni2}_t&=\int_0^t  I_i \ h(X^n_s,X^n_{n(s)})\mu(X^n_{n(s)}) \sqrt{n}\Delta s^{(n)} dW_s,\\
F^{ni3}_t&=\int_0^t  I_i \ h(X^n_s,X^n_{n(s)})\sigma(X^n_{n(s)}) \sqrt{n} \Delta W_s^{(n)} dW_s- \frac{\sqrt{2}}{2} \int _0^t  I_i\sigma(X_t) \sigma^\prime(X_t) dB_s.
\end{align*}
Note that 
\begin{align}\label{eqGF}
\begin{split}
&\sum_{i=1}^2 \sum_{j=1}^3 (G^{nij}+F^{nij})\\
&=U^n-\Big\{\int _0^t \mu^\prime(X_t) R_tds+\int _0^t \sigma^\prime(X_s) R_sdW_s +\frac{\sqrt{2}}{2} \int _0^t \sigma(X_s) \sigma^\prime(X_s) dB_s\Big\}.
\end{split}
\end{align}
Our goal is to show that each term of $G^{nij},F^{nij}$ converges to a $0$ process on $[0,T]$ in distribution.


Consider term $G^{n11}$. Since $(X^n,X^n_{n(.)},U^n) \overset{a.s.}{\to}(X,X,R)$ as $n \to \infty$, and $\mu$ differentiable on $A^c$, for each $\omega \in \A^c$, 
 \begin{align*}
I_{\{X_t \in A^c \}}g(X^n_{n(t)},X_t) U^{n}_t {\rightarrow} I_{\{X_t \in A^c \}}\mu^\prime(X_t) R_t, 
\quad \text{pointwise in $t$}.
\end{align*}
Since $U^n$ is a continuous process and the convergence is uniform, 
it's limit $R$ will be continuous as well, and moreover 
$$
R_T^*(\omega)=\underset{0\leq s \leq T}{\sup}\abs{R_s}(\omega)<\infty,\,\,\,
\sup\limits_n\underset{0<s\leq T}{\sup} \abs{U^{n}_s}(\omega)<\infty.
$$
By the globally Lipschitz condition on $\mu$,
\begin{align*}
\abs{g(X^n_{n(t)},X_t) U^{n}_t-\mu^\prime(X_t) R_t } \leq K ( \underset{0<s\leq T}{\sup} \abs{U^{n}_s}+\underset{0<s\leq T}{\sup} \abs{R_s} ).
\end{align*}
Applying the dominated convergence theorem, $G^{n11}$ converges to 0 uniformly almost surely on $[0,T]$. 

Consider term $G^{n21}$. By (\ref{bdd1}) and (\ref{bd2}) there exists $C_1>0$
\begin{align}\label{G21-1}
\begin{split}
\mathbb{E} (G^{n21})&\leq \mathbb{E} \big[ \int_0^T  I_{\{s: X_s \in  A \}} K \big\{\abs{U^{n}_s}+\abs{ R_s} \big\}ds \big] \\
& \leq K\Big(\mathbb{E} \Big[  \big(\underset{0<s\leq T}{\sup} \abs{U^{n}_s}+  \underset{0<s\leq T}{\sup} \abs{R_s}\big)^2\Big ]\Big)^{\frac{1}{2}} \Big(\mathbb{E}\Big[ \int_0^T  I_{\{s: X_s \in  A \}} ds\Big]\Big)^{\frac{1}{2}} \\
& \leq C_1 \mathbb{E}\Big[ \int_0^T  I_{\{s: X_s \in  A \}} ds\Big].
\end{split}
\end{align}
By Corollary 3.8 in Chap 7 of Revuz and Yor \cite{RD2013}, let $\T^a, \T^b$ be hitting time of $X_s$ and $a<b$, then
\begin{align*}
\mathbb{E}\Big[ \int_0^{\T^a\wedge \T^b}  I_{\{s: X_s \in  A \}} ds\Big]= \int_0^{{\T^a}\wedge{\T^b}} G_I(x_0,y)I_{\{y \in  A \}}m(dy),
\end{align*}
where 
\begin{align*}
s(x)&=\int_c^x \exp \Big (-\int_c^y 2 \mu(z) \sigma^{-2}(z)dz  \Big ) dy, \quad \forall c\in \mathbb R; \\
G_I&=\begin{cases}
      \frac{(s(x)-s(a))(s(b)-s(y))}{s(b)-s(a)}, & \ a\leq x\leq y\leq b, \\
      \frac{(s(y)-s(a))(s(b)-s(x))}{s(b)-s(a)}, & \ a\leq y\leq x\leq b,\\
      0, & \ otherwise;
      \end{cases}  \\
m(dx)&=\frac{2}{s^\prime(x) \sigma^2(x)}dx.
\end{align*}
Recall that $s$ is the scale function, $G_I$ is the Green function and $m(dy)$ is the speed measure.

By the boundedness of $\mu,\sigma$, we have $G_I(x_0,y)$ and $\frac{2}{s^\prime(x) \sigma^2(x)}$ are bounded. Since $A$ has Lebesgue measure $0$
\begin{align*}
\mathbb{E}\Big[ \int_0^{T\wedge{\T^a}\wedge{\T^b}}  I_{\{s: X_s \in  A \}} ds\Big]\le 
\mathbb{E}\Big[ \int_0^{{\T^a}\wedge{\T^b}}  I_{\{s: X_s \in  A \}} ds\Big]=0.
\end{align*}  
Let $a\to -\infty, b\to \infty$, and apply Fatou's lemma, 
\begin{align}\label{bd3}
\mathbb{E}\Big[ \int_0^{T}  I_2 ds\Big]=\mathbb{E}\Big[ \int_0^{T}  I_{\{s: X_s \in  A \}} ds\Big]=0
\end{align}
Together with (\ref{G21-1}), we have $G^{n21}=0$. 

Consider $F^{n11}$, from the Burkholder-Davis-Gundy inequality, there exists $C_3>0$ s.t.
\begin{align}\label{eqFn11}
\begin{split}
\mathbb E \Big[\underset{0<s\leq T}{\sup} \abs{ F^{n11}_t } \Big ]&\leq 
C \ \mathbb E \Big[\Big( \int_0^T  I_1 \big(h(X^n_{n(s)},X_s) U^{n}_s-\sigma^\prime(X_s) R_s \big)^2ds \Big)^{\frac{1}{2}} \Big] \\
& \leq C_3  \mathbb E \Big[ \Big(\int_0^T  I_1 \big \{h(X^n_{n(s)},X_s) (U^{n}_t-R_t)\big\}^2 ds \Big)^{\frac{1}{2}}\Big]\\
&+C_3\mathbb E \Big[ \Big(\int_0^T  I_1 \big\{R_t  (h(X^n_{n(s)},X_s)-\sigma^\prime(X_t))\big\}^2 ds \Big)^{\frac{1}{2}}\Big].
\end{split}
\end{align}
Consider the first term on the right side of (\ref{eqFn11}). Since $|h|\le K$,
\begin{align*}
&\mathbb E\Big[ \Big(\int_0^T   I_1\big \{h(X^n_{n(s)},X_s) (U^{n}_s-R_s)\big\}^2 ds \Big )^{\frac{1}{2}} \Big]
\leq  K\, T\,\mathbb E \Big(\underset{0<s\leq T}{\sup} \abs{U^{n}_s-R_s}\Big).
\end{align*}
On the one hand $\underset{0<s\leq T}{\sup} \abs{U^{n}_s-R_s}\overset{a.s.}{\to} 0$. On the other by (\ref{bdd1}) and (\ref{bd2}), 
we get 
\begin{align*}
\underset{n}{\sup} \Big(\mathbb E\Big[ \underset{0<s\leq T}{\sup} \abs{U^{n}_s-R_s}^2\Big]\Big)< \infty,
\end{align*}
which gives a uniform integrability condition to ensure
\begin{align*}
\underset{n\to \infty}{\lim} \mathbb E \Big(\underset{0<s\leq T}{\sup} \abs{U^{n}_s-R_s}\Big)=0.
\end{align*}
Thus the first term on the right side of (\ref{eqFn11}) converges to $0$. For the second term, an application of the Hölder's inequality gives
\begin{align*}
&\mathbb E  \Big[\Big(\int_0^T  I_1 \big\{R_s  (h(X^n_{n(s)},X_s)-\sigma^\prime(X_s))\big\}^2 ds\Big)^{\frac{1}{2}}\Big]\\
&\leq  \mathbb E\Big[ \underset{0<s\leq T}{\sup} \abs{R_s} \Big( \int_0^T  I_1 \big\{h(X^n_{n(s)},X_s)-\sigma^\prime(X_s)\big\}^2 ds\Big)^{\frac{1}{2}}\Big]\\
&\leq  \Big(\mathbb E\Big[ \underset{0<s\leq T}{\sup} \abs{R_s}^2\Big]\Big)^{\frac{1}{2}}
 \Big( \mathbb E \Big[ \int_0^T  I_1 \big\{h(X^n_{n(s)},X_s)-\mu^\prime(X_s)\big\}^2 ds \Big]\Big)^{\frac{1}{2}}.
\end{align*}
For each $\omega \in \A^c$, we have 
$I_{\{t: X_t \in  A^c \}} h(X^n_{n(t)},X_t)\overset{a.s.}{\to} \ I_{\{t: X_t \in  A^c \}}\sigma^\prime(X_t)$ pointwise in $t$, and 
$\abs{h(X^n_{n(s)},X_s)}, \abs{\mu^\prime(X_s)}$ are uniformly bounded by $K$. 

From the dominated convergence theorem, 
\begin{align*}
\underset{n\to \infty}{\lim}  \mathbb E \Big( \int_0^T  I_{\{s: X_s \in \mathbb{R}\cap A^c \}} 
\big\{h(X^n_{n(s)},X_s)-\sigma^\prime(X_s)\big\}^2 ds \Big)=0.
\end{align*}
With (\ref{bd2}), we have the second term of right side of (\ref{eqFn11}) also converges to $0$. Thus
\begin{align*}
\underset{n\to \infty}{\lim} \mathbb E \underset{0<s\leq t}{\sup} \abs{ F^{n11}_t } =0.
\end{align*}
For the term $ F^{n21}$, we would like to prove
\begin{align}\label{lim1}
\underset{n\to \infty}{\lim} \mathbb E \underset{0<s\leq T}{\sup} \abs{ F^{n21}_s } =0.
\end{align}
Similarly to the analysis of $F^{n11}$, to prove (\ref{lim1}) we are only left to prove 
 \begin{align*}
\underset{n\to \infty}{\lim} \mathbb E \big( \int_0^T  I_2 \big\{h(X^n_{n(s)},X_s)-\sigma^\prime(X_s)\big\}^2 ds \big)=0, 
\end{align*}
which is implied by (\ref{bd3}) and boundedness of $\abs{h(X^n_{n(s)},X_s)}, \abs{\sigma^\prime(X_s)}$ .\\

Consider the terms $G^{n12},G^{n13},G^{n22},G^{n23},F^{n12},F^{n22}$ all of which converge to the constant 
process 0 almost surely uniformly on $[0,T]$ because $g,h,\mu,\sigma$ are bounded and 
$(Z^{n11},Z^{n12},Z^{n21})\overset{a.s.}{\to}(0,0,0)$ uniformly on $[0,T]$. \\
For dealing with the last two terms $F^{n13}$ and $F^{n23}$, we first define $\tilde{F}^{n13}$ as 
 \begin{align*}
\tilde F^{n13}_t&=\int_0^t  I_1\ \big\{ h(X^n_s,X^n_{n(s)})\sigma(X^n_{n(s)})-\sigma^\prime(X_s)\sigma(X_s)  \big\} \sqrt{n} \Delta W_s^{(n)} dW_s.
 \end{align*}
From the Burkholder-Davis-Gundy inequality, there exists $C_3>0$ such that
\begin{align*}
\mathbb E [\underset{0<s\leq T}{\sup} \abs{ \tilde F^{n13}_s } ]\leq 
C_3 \ \mathbb E \Big[\Big( \int_0^T  I_1 \big(h(X^n_s,X^n_{n(s)})\sigma(X^n_{n(s)})-\sigma^\prime(X_s)\sigma(X_s) \big)^2 (\sqrt{n} \Delta W_s^{(n)})^2 ds \Big)^{\frac{1}{2}} \Big] .
\end{align*}  
Applying Cauchy-Schwarz inequality to the right side, there exists $C^\prime_4, \ C_4>0$ s.t.
\begin{align*}
&\mathbb E [\underset{0<s\leq T}{\sup} \abs{ \tilde F^{n13}_s } ]\\
&\leq C^\prime_4\mathbb E\Big[ \Big(  \int_0^T  I_1 \big(h(X^n_s,X^n_{n(s)})\sigma(X^n_{n(s)})-\sigma^\prime(X_s)\sigma(X_s) \big)^4 ds  \Big)^{\frac{1}{4}}  \Big(  \int_0^T  I_1 \big(  \sqrt{n} \Delta W_s^{(n)} \big)^4 ds \Big)^{\frac{1}{4}} \Big]\\
&\leq C^\prime_4\Big[ \mathbb E \Big(  \int_0^T  I_1 \big(  \sqrt{n} \Delta W_s^{(n)}  \big)^4 ds  \Big)^{\frac{1}{2}} \Big]^{\frac{1}{2}}
\Big[ \mathbb E \Big(  \int_0^T  I_1 \big(h(X^n_s,X^n_{n(s)})\sigma(X^n_{n(s)})-\sigma^\prime(X_s)\sigma(X_s) \big)^4 ds  \Big)^{\frac{1}{2}} \Big]^{\frac{1}{2}}  \\
&\leq C_4 \Big[ \mathbb E \Big(  \int_0^T  I_1 \big(h(X^n_s,X^n_{n(s)})\sigma(X^n_{n(s)})-\sigma^\prime(X_s)\sigma(X_s) \big)^4 ds  \Big)^{\frac{1}{2}} \Big]^{\frac{1}{2}}.
\end{align*}  
Since $h,\sigma,\sigma^{\prime}$ are bounded, by the dominated convergence theorem,
\begin{align*}
\underset{n\to \infty}{\lim}\mathbb E \Big[\underset{0<s\leq t}{\sup} \abs{ \tilde F^{n13}_t } \Big]
=\underset{n\to \infty}{\lim}\mathbb E \Big(  \int_0^T  I_1 \big(h(X^n_s,X^n_{n(s)})\sigma(X^n_{n(s)})-\sigma^\prime(X_s)\sigma(X_s) \big)^4 ds  \Big)^{\frac{1}{2}}=0.
\end{align*}
Thus $\tilde F^{n13} \overset{L^1}{\to} 0$ uniformly on $[0,T]$. We define  $\bar{F}^{n13}$ as 
 \begin{align}
\bar F^{n13}_t
&=\int_0^T  I_1\ \sigma^\prime(X_s)\sigma(X_s) dZ^{n22}
- \int_0^T  I_1\ \sigma^\prime(X_s)\sigma(X_s) dB_s.
\end{align}
Since $Z^{n22}\overset{a.s.}{\to} B_s $ uniformly on $[0,T]$ and $Z^{n22}$ is a good sequence, the result on convergence in probability of stochastic integrals in Protter and Kurtz \cite{KT1991-1} leads to $\bar F^{n13} \overset{p}{\to} 0$ uniformly on $[0,T]$. As $F^{n13}=\tilde F^{n13}+ \bar F^{n13}$, 
$ F^{n13} \overset{p}{\to} 0$.\\
For the last term $F^{n23}$, applying the Burkholder-Davis-Gundy inequality first, then using the same technique as in bounding $\tilde F^{n13}$, together with (\ref{bd3}), give 
\begin{align*}
&\underset{n\to \infty}{\lim}\mathbb E \Big[\underset{0<s\leq T}{\sup} \abs{\int_0^T  I_2 \ h(X^n_s,X^n_{n(s)})\sigma(X^n_{n(s)}) \sqrt{n} \Delta W_s^{(n)} dW_s } \Big]=0,\\
&\underset{n\to \infty}{\lim} \mathbb E \Big[\underset{0<s\leq T}{\sup} \abs{\int _0^T  I_2\sigma(X_s) \sigma^\prime(X_s) dB_s } \Big]=0.
\end{align*}
Thus $F^{n23} \overset{L^1}{\to} 0$ uniformly on $[0,T]$. Each of the $G$ and $F$ terms converges to $0$ uniformly on $[0,T]$ either almost surely or in $L^1$ or in probability.
Then, by (\ref{eqGF}), $U^n \overset{p}{\to} \tilde R$ uniformly on $[0,T]$, where
\begin{align*}
\tilde R_t=\int _0^t \mu^\prime(X_s) R_sds+\int _0^t \sigma^\prime(X_s) R_sdW_s +\frac{\sqrt{2}}{2} \int _0^t \sigma(X_s) \sigma^\prime(X_s) dB_s .
\end{align*}
Since also $U^n\overset{a.s}{\to} R$ on $[0,T]$, the two limits must equal each other, and $R$ follows
\begin{align}\label{errorsde}
R_t=\int _0^t \mu^\prime(X_t) R_tds+\int _0^t \sigma^\prime(X_s) R_sdW_s +\frac{\sqrt{2}}{2} \int _0^t \sigma(X_s) \sigma^\prime(X_s) dB_s .
\end{align}
This concludes the proof.
\end{proof}

\begin{remark}\label{r3}
Both in Kurtz and Protter~\cite{KT1991-2} and Neuenkirch and Zähle \cite{AN}, $\mu$ and $\sigma$ are assumed to be $
\mathcal{C}^1$. Since Lipschitz continuity does not imply differentiability, the key part in proof of Theorem~\ref{th.3-2} is to show that the time the weak limit error process spends on the set where $\mu$ and $\sigma$ are not differentiable has Lebesgue measure 0. 

\end{remark}
\section{Study of The Normalized Limit Error Process}
With the weak limit of normalized error process for the Euler scheme being derived, we are interested to further analyze its properties. Though Kurtz and Protter~\cite{KT1991-2} derived the form of the normalized error process of the Euler scheme under the condition that the coefficients are $\mathcal{C}^1$ and bounded, its properties have barely been studied in previous work. In this section, we focus on the mean, variance and martingality of the limit error process under the globally Lipschitz condition. The locally Lipschitz case is more complicated and is studied through examples as well.

\subsection{The Globally Lipschitz Case}
\begin{theorem}{\label{th.4-1}}
When $\mu$ and $\sigma$ are globally Lipschitz, for the normalized error process $U_n=\sqrt{n}(X^n-X)$ from the continuous Euler scheme, there exists $0<C_t<\infty$, where $C_t$ increasing with $t$, such that   
 \begin{align*}
\mathbb{E} [U_t^2] \leq \mathbb{E} [ U^{*2}_t] \leq C_t, 
 \end{align*}
where $U^*_t=\underset{0\leq s\leq t} {\sup} \abs{U_s}$. 
Furthermore when $\mu^\prime=0$, $U$ is a square integrable martingale.
\end{theorem}
\begin{proof}
Since $U^n\Rightarrow U$ uniformly on $[0,T]$, we have $\forall \ t \in [0,T]$, $U^{n*}_t\Rightarrow U^*_t$. 
When $\mu$ and $\sigma$ are both globally Lipschitz, from Kloeden \cite{KP1999} proof of Theorem 10.2.2, there exists a $C_t$, increasing with t, such that
\begin{align}\label{bd1}
\underset{n}{\sup} \ \mathbb E [(U^{n*}_t)^2]< C_t.
\end{align} 
Without loss of generality we can assume there exists a subsequence $(U^{n_k})^2\overset{a.s.}{\to} U^2$ uniformly on $[0,T]$. Since $(U^{n_k})^2\geq 0$, from Fatou's lemma  
 \begin{align*}
\mathbb{E} [U_t^2] \leq \mathbb{E}[U^{*2}_t]\leq \underset{k\rightarrow \infty} {\liminf} \ \mathbb{E}[(U_t^{n_k*})^2]\leq C_t.
\end{align*}
%
When $\mu^\prime=0$, there is no drift term in (\ref{Sde4}). Thus $U$ is a local martingale. We also have a bound for the expectation of the quadratic variation of $U_t$. Since $\mu, \sigma$ are globally Lipschitz, it is known that $\mathbb{E}(X_t^2)<\infty, \forall \ t \in [0,T]$.
Let $K$ be the Lipschitz coefficients for $\mu$ and $\sigma$, then
 \begin{align*}
 \mathbb{E}(\langle U,U\rangle_t)&= \mathbb{E}\Big[\int_0^t \{\sigma^{\prime 2} (X) U_s^2 +\sigma^{\prime 2} (X) \sigma^2(X_s)\} ds \Big ]\\
 &\leq  \int_0^t  K^2 \mathbb{E}(U_s^2) ds + \int_0^t  K^4\mathbb{E}(X_s^2) ds <\infty.
\end{align*}
$U_s$ is a local martingale with finite expected quadratic variation. From Corollary 3 in page 73 in Protter \cite{PP1990}, we conclude it is a martingale when $\mu^\prime=0$. 
\end{proof}

\subsection{The Locally Lipschitz Case and Examples}\label{ss1a}
\subsubsection{ The Inverse Bessel Process}
When $\mu$ and $\sigma$ are only locally Lipschitz, the finiteness of the second moment of the corresponding $U_t$ may not hold. Theorem  \ref{th.4-1} cannot be extended to the locally Lipschitz plus no finite explosion time case. One example is the inverse Bessel process, which is a solution to the SDE 
 \begin{align*}
dX_t=X_t^2dW_t, \ \ X_0>0. 
\end{align*}
The coefficient $\sigma(x)=x^2$ is locally Lipschitz and $X$ has no finite explosion. From Theorem \ref{th.3-2}, the error process $U^n_t=\sqrt{n}(X_t^n-X_t)$ converges in distribution uniformly to $U_t$ on $[0,T]$. $U_t$ is the solution to
 \begin{align*}
dU_t=2 X_t U_tdW_t+{\sqrt{2}}X_t^{3}dB_t,
\end{align*}
where $B$ is a Brownian motion independent of $W$. 
 \begin{align*}
 \mathbb{E}(U_t^2)&= \mathbb{E} \Big ( 2\int_0^t X_s U_sdW_s + {\sqrt{2}} \int_0^t  X_s^{3}dB_s \Big )^2\\
 &=4\mathbb{E} \Big ( \int_0^t X_s U_sdW_s\Big )^2  + 2\mathbb{E}   \Big (\int_0^t  X_s^{6}ds \Big )
 \end{align*}
Since the inverse Bessel process can also be represented as the inverse of the norm of a three dimensional Brownian motion starting from
$(1,0,0)$, its explicit distribution can be obtained (for example see \cite{FH2011}). A calculation shows if $X_0>0$, then $\forall t>0$, $\mathbb{E}X_t^6=\infty$. This gives $\mathbb{E}(U^2_t)=\infty$ and $\mathbb{E}(U^{*2}_t)=\infty$. This indicates that under the locally Lipschitz condition,
the asymptotic distribution for the normalized error process might have a larger tail probability than in the globally Lipschitz case. 
\subsubsection{ The CIR process}
There are however examples with $\mu$ and $\sigma$ only locally Lipschitz, and $U_t$ still has a finite second moment. We look at the Cox-Ingersoll-Ross model (or CIR model) which  is often used to describe the evolution of interest rates. The CIR process follows the SDE
 \begin{align}\label{cir}
dX_t=(a-bX_t)dt+\sigma \sqrt{X_t} dW_t, \ \ X_0>0, a>0. 
\end{align}
The coefficient function $\sigma \sqrt{X_t}$ is only locally Lipschitz. The true solution to (\ref{cir}) remains always positive, but the numerical solution from the Euler scheme may go negative. Thus the Euler scheme is not well defined for solving (\ref{cir}). We use the same trick due to Bossy at al  \cite{Bossy2008}, replacing the Euler scheme by a symmetrized Euler scheme. Let $U^n$ be the sequence of approximate normalized errors from the 
symmetrized Euler scheme solving (\ref{cir}). By Theorem 2.2 in Berkaoui, Bossy and Diop \cite{Bossy2008}, 
there exists a $C_t$, increasing with t, such that
\begin{align}\label{bd1}
\underset{n}{\sup} \ \mathbb E [(U^{n*}_t)^2]< C_t,
\end{align} 
if the following condition holds
\begin{align*}
\frac{\sigma^2}{8}\big( \frac{2a}{\sigma^2}-1\big)^2>\mathcal K(8), \ \text{with} \ \mathcal K(p)=\max \{b(4p-1),(2\sigma (2p-1) )^2\}.
\end{align*} 
Since the symmetrized Euler scheme is local and the true solution never hits $0$ or $\infty$ in finite time, it can be shown that $U^n \Rightarrow U$ as $n\to \infty$ on any finite time interval. The weak limit $U$ has the same form as in Theorem \ref{th.3-2}; Indeed, it solves the SDE below.
\begin{align*}
dU_t=-bU_tdt+\frac{\sigma U_t}{2 \sqrt{X_t}}dW_t+\frac{\sqrt{2}}{2}\sigma^2 dB_t.
\end{align*} 
With (\ref{bd1}), applying Fatou's Lemma, we have 
 \begin{align*}
\mathbb{E} [U_t^2] \leq \mathbb{E} [ U^{*2}_t] \leq C_t. 
 \end{align*}
 The inverse Bessel and CIR examples show that the finiteness of the second moment of the normalized error process for the Euler scheme (or modified Euler scheme in order for the scheme to be well defined) under the locally Lipschitz situation is more complicated than the globally Lipschitz situation. 
\section{Approximation of Expectations of Functionals}
In applications, the convergence of expectations of functionals (also called weak convergence in existing literature) of the Euler scheme is important.
 To avoid confusion, in this section weak convergence means the convergence of expectations of functionals unless further specified. 
We are interested in the rate of convergence for
$\mathbb{E}[g(X_T^n)]-\mathbb{E}[g(X_T)]$ to $0$, as $n$ goes to infinity. When $\mu$ and $\sigma$ are only 
assumed to be locally Lipchitz, inferred from Hutzenthaler, Jentzen and Kloeden \cite{MH2011}, even for $g$ with linear growth, weak convergence in 
the sense of expectations of functionals may not hold. As a compromise, in this section we assume $g$ is Lipschitz and 
bounded, and give upper bound for the weak convergence rate with the no finite explosion condition and some other mild 
conditions on the SDE (\ref{Sde1}). Before we deal with the locally Lipschitz case, we need the following Proposition.  

In Proposition \ref{weak1}, inequality (\ref{weak_1}) can be inferred from Kloeden \cite{KP1999} page 343 proof of 
Theorem 10.2.2 in chapter 10, or Theorem 4.4 in H. Desmond and X.Mao \cite{HD2002}. 

\medskip

\begin{proposition}\label{weak1}
Consider SDE (\ref{Sde1}), if $\mu$ and $\sigma$ are globally Lipschitz with Lipschitz coefficient as K, then 
 for all $T>0$ there exists $c>0$ not depending on $K$ increasing with $T$, such that for all $n \in \mathbb{N}^{+}$,  
\begin{align} \label{weak_1}
 \mathbb{E}[ \underset{0\leq s \leq T} {\sup}(X^n_s-X_s)^2]\leq \frac{ \exp(cK)}{n},
\end{align}
and for all $\gamma \in [0,\frac{1}{2})$
\begin{align}\label{weak_2}
 \mathbb P(\underset{0\leq s \leq T} {sup} \abs{X^n_s-X_s}>n^{-\gamma})\leq \exp(cK)n^{-1+2\gamma}.
\end{align}
\end{proposition}
\begin{proof}
From Theorem 4.4 in H. Desmond and X.Mao \cite{HD2002}, with the globally Lipschitz condition, 
for any $\delta>0$ there exists universal constant $C$ and $A$ independent of $n$ such that
\begin{align*}
 \mathbb{E}[ \underset{0\leq s \leq T} {\sup}(X^n_s-X_s)^2]\leq 
 \frac{C}{n}(K^2+1)\exp\{4K(T+4)\}+\frac{A}{n}.
\end{align*}
Thus (\ref{weak_1}) holds. Applying the Chebyshev's inequality gives (\ref{weak_2}).
\end{proof}
\begin{theorem}\label{th3}
Consider the SDE (\ref{Sde1}). If $\mu, \sigma$ are locally Lipschitz and we assume the Lipschitz constant
has at most polynomial growth with exponent $a \in \mathbb{R}^+$, that is, for all $x,y \in \mathbb{R}$
$$
|\mu(x)-\mu(y)|+|\sigma(x)-\sigma(y)|\le (\max\{|x|,|y|\}+K)^a \, |x-y|,
$$
where $K$ is a constant.
We assume there exists $ \kappa,\nu>0$, such that for all $x>0$ 
\begin{align*} 
\mathbb P(X_T^*>x)\leq \kappa x^{-\nu}.
\end{align*}
Then, there exists a finite constant $C=C(\kappa,\nu,K,a,|x_0|)$ such that
for any Lipschitz and bounded function $g$ and for all $n>1$
\begin{align*} 
\abs{\mathbb{E}g(X_T^n)-\mathbb{E}g(X_T)}<C(\|g\|_\infty+G+1)\big(\log n\big)^{-\frac{\nu}{a}},
\end{align*}
where $G$ is the Lipschitz constant of $g$.
\end{theorem}
\begin{proof}
For a fixed $m>\abs{x_0}$, define $\mu^{(m)}, \sigma^{(m)}$ and $Y^{(m)}$ as in Proposition \ref{th.CVP}. 
Let $\T^m(Z)=\inf\{ t\ge 0: |Z_t|>m\}$.
Since the Euler scheme is local,
\begin{align*}
&\T^m(X(x_0,\mu,\sigma,W))=\T^m(X(x_0,\mu^{(m)},\sigma^{(m)},W)),\\
& \T^m(X^n(x_0,\mu,\sigma,W))=\T^m(X^n(x_0, \mu^{(m)},\sigma^{(m)},W)). 
\end{align*}
Let $\theta^m= \T^{m+2}(X(x_0,\mu,\sigma,B)) \wedge \T^{m+2}(X^n(x_0,\mu,\sigma,B))$. Then 
\begin{align*} 
\abs{\mathbb{E}g(X_T^n)-\mathbb{E}g(X_T)} \le 
\abs{\mathbb{E}g(X_{T\wedge \theta^m }^n)-\mathbb{E}g(X_{T\wedge \theta^m })}+2\|g\|_\infty\mathbb P(\theta^m<T)
\end{align*}
Let $G$ be a Lipschitz constant for $g$. Then by (\ref{weak_1}) in Proposition \ref{weak1}, and since
the Lipschitz constant of $\mu,\sigma$ on $[-(m+2),m+2]$ is bounded by $(m+K+2)^a$
\begin{align*} 
\abs{\mathbb{E}g(X_{T\wedge \theta^m }^n)-\mathbb{E}g(X_{T\wedge \theta^m })} \leq 
G( \mathbb{E} \abs{X_{T\wedge \theta^m }^n-X_{T\wedge \theta^m }}^2 )^{\frac{1}{2}} \leq 
G \exp\left(\frac{c}{2} (m+K+2)^a\right)n^{-\frac{1}{2}},
\end{align*}
where $c$ is the constant given in Proposition \ref{weak1}.
By the distribution assumption on $X^*$ and (\ref{weak_2}) in Proposition \ref{weak1} with $\gamma=0$, 
\begin{align*} 
&\mathbb P(\theta^m<T)\le\mathbb P( \T^{m+2}(X(x_0,\mu,\sigma,W)) <T)+\mathbb P( \T^{m+2}(X^n(x_0,\mu,\sigma,W)) <T)\\
&\le 2\mathbb \mathbb P( \T^{m+1}(X) \le T)+
\PP(\T^{m+1}(X) >T,\, \T^{m+2}(X^n) <T)\\
&=2\mathbb \mathbb P( \T^{m+1}(X) \le T)+
\PP(\T^{m+1}(X^{(m+2)}) >T,\, \T^{m+2}(X^{n,(m+2)}) <T)\\
&\le 2\mathbb \mathbb P( \T^{m+1}(X) <T)+ \exp(c(m+K+2)^a)n^{-1}\\
&\le 2\kappa (m+1)^{-\nu}+ \exp(c(m+K+2)^a)n^{-1}.
\end{align*} 
For $m\ge 0$, we have $(m+1)^{-\nu}\le (K+2)^\nu (m+K+2)^{-\nu}$. Thus, we get
$$
\begin{array}{ll}
\abs{\mathbb{E}g(X_T^n)-\mathbb{E}g(X_T)} \le &2 \|g\|_\infty \left[2\kappa\,(K+2)^\nu (m+K+2)^{-\nu}+
\exp(c(m+K+2)^a)n^{-1}\right]
\\&+G \exp\left(\frac{c}{2} (m+K+2)^a\right)n^{-\frac{1}{2}}.
\end{array}
$$
Take $n= (m+K+2)^{2\nu}\exp(c (m+K+2)^a)$ to get
$$
\abs{\mathbb{E}g(X_T^n)-\mathbb{E}g(X_T)} \le (2\|g\|_\infty[(K+2)^\nu 2\kappa+1]+G) (m+K+2)^{-\nu}.
$$
Notice that 
$$ 
\log n\le c(m+K+2)^a+2\nu \log(m+K+2)\le (c+2\nu/a) (m+K+2)^a,
$$
which implies that
$$
\abs{\mathbb{E}g(X_T^n)-\mathbb{E}g(X_T)} \le (2\|g\|_\infty[(K+2)^\nu 2\kappa+1]+G) \left(\frac{\log n}{c+2\nu/a}\right)^{-\frac{\nu}{a}}.
$$
The result follows from this estimation. 
Finally, notice we have assumed $m\ge |x_0|$, which imposes that $n\ge n_0$ has to be large enough, for example
$$
\log n_0\ge c (|x_0|+K+2)^a+2\nu \log(|x_0|+K+2).
$$

\end{proof}

As an example we consider the constant elasticity of variance process which follows the following SDE
\begin{align*} 
dS_t=bS_t^\beta dW_t, \ S_0>0
\end{align*}
When $ \beta>1$, the solution to the above SDE is strict local martingale and is used for detecting asset bubbles. 
By a result of A. N. Borodin and P. Salminen \cite{BS2012}, chapter 4.6, $\forall x>S_0, \ T>0$
\begin{align*} 
P(S_T^*>x)<\frac{S_0}{x}.
\end{align*}
Thus, there exists a constant $C>0$, such that for all $g : \mathbb R \rightarrow \mathbb{R} $, 
bounded and Lipschitz and for all $n>1$
\begin{align*} 
 \abs{\mathbb{E}g(S_T^n)-\mathbb{E}g(S_T)}<C(\|g\|_\infty+G+1)\big(\log n\big)^{-\frac{1}{(\beta-1)}}
\end{align*}

\bigskip

\begin{remark}\label{r5} 
The above example of the CEV process for $\beta >1$ illustrates the weakness of the result of Theorem~\ref{th3}. The rate of convergence is so slow as to be essentially useless in practice. It is our hope that future research will illustrate methods that will permit a more practically useful analysis of the rate of convergence. This seems far away at this point. 
\end{remark}



\begin{thebibliography}{40}


\bibitem{JA2009} J. Arnulf and P. E. Kloeden,  The numerical approximation of stochastic partial differential equations. \textit{Milan Journal of Mathematics} 77.1: 205-244, 2009.
\bibitem{BV1995} V. Bally and D. Talay, The Euler scheme for stochastic differential equations: error analysis with Malliavin calculus. \textit{Mathematics and computers in simulation}  38.1-3: 35-41, 1995.
\bibitem{BV1996-1}V. Bally and D. Talay, The law of the Euler scheme for stochastic differential equations. \textit{Probability theory and related fields} 104.1: 43-60, 1996.
\bibitem{BV1996-2} V. Bally and D. Talay, The law of the Euler scheme for stochastic differential equations: II. Convergence rate of the density.  \textit{Monte Carlo Methods and Applications} 2.2: 93-128, 1996.
\bibitem{Bass} R. F. Bass and E. Pardoux, Uniqueness for diffusions with piecewise constant coefficients. \textit{Probability Theory and Related Fields} 76.4: 557-572, 1987.
\bibitem{Bossy2008} A. Berkaoui, M. Bossy and A. Diop, Euler scheme for SDEs with non-Lipschitz diffusion coefficient: strong convergence. \textit{ESAIM: Probability and Statistics}, 12, 1-11, 2008.
\bibitem{BS2012} A.N. Borodin and P. Salminen, \emph{Handbook of Brownian Motion-Facts and Formulae, Second Edition}, Birkh\"auser, 2015.
\bibitem{Bossy} M. Bossy, A. Diop. An efficient discretisation scheme for one dimensional SDEs with a diffusion coefficient function of the form $\abs{x}^ a$, $a$ in $[1/2, 1)$ (Doctoral dissertation, INRIA).
\bibitem{EH1985} H. Engelbert and W. Schimidt, On one-dimensional stochastic differential equations with generalized drift.\textit{ Stochastic Differential Systems Filtering and Control}: 143-155, 1985.
\bibitem{FH2011}  H. Föllmer and P. Protter, Local martingales and filtration shrinkage. \textit{ESAIM: Probability and Statistics} 15: S25-S38, 2011.
\bibitem{GI} I, Gyöngy. A note on Euler's approximations. \textit{Potential Analysis}, May 1;8(3):205-16, 1998 .


\bibitem{HD2002} D. J. Higham, X. Mao and A.M.Stuart,  Strong convergence of Euler-type methods for nonlinear stochastic differential equations. \textit{SIAM Journal on Numerical Analysis} 40.3: 1041-1063, 2002.
\bibitem{MH2011} M. Hutzenthaler,  A. Jentzen and P. E. Kloeden, Strong and weak divergence in finite time of Euler's method for stochastic differential equations with non-globally Lipschitz continuous coefficients. \textit{Proceedings of the Royal Society of London A: Mathematical, Physical and Engineering Sciences. Vol. 467. No. 2130.} The Royal Society, 2011.
\bibitem{MH2015} M. Hutzenthaler and A. Jentzen, Numerical approximations of stochastic differential equations with non-globally Lipschitz continuous coefficients. Vol. 236. No. 1112. American Mathematical Society, 2015.
\bibitem{JJ1998} J. Jacod and P. Protter, Asymptotic error distributions for the Euler method for stochastic differential equations. \textit{Annals of Probability}: 267-307, 1998.
\bibitem{JJ2011} J. Jacod and P. Protter, \textit{Discretization of processes}. Vol. 67. Springer Science and Business Media, 2011.
\bibitem{KI2012} I. Karatzas, S. Shreve, \textit{Brownian motion and stochastic calculus, Second Edition}, Springer, NewYork, 2012.
\bibitem{KP1999} P. E. Kloeden and E. Platen, \textit{Numerical solution of stochastic differential equations}, Springer Verlag, New York,1999.
\bibitem{KH1994} A. Kohatsu-Higa, and P. Protter, The Euler scheme for SDE's driven by semimartingales.  \textit{Pitman research notes in mathematics series}: 141-151, 1994.
\bibitem{KT1991-1} T. G. Kurtz and P. Protter, Weak limit theorems for stochastic integrals and stochastic differential equations. \textit{The Annals of Probability}: 1035-1070, 1991.
\bibitem{KT1991-2} T. G. Kurtz and P. Protter, Wong-Zakai corrections, random evolutions and numerical schemes for SDE's.\textit{ Stochastic Analysis}: 331-346, 1991.


\bibitem{MG2002} G. Marion X. Mao and E. Renshaw, Convergence of the Euler scheme for a class of stochastic differential equation. \textit{International Mathematical Journal} 1.1: 9-22, 2002.
\bibitem{AM} A. Mijatovic, M. Urusov, On the martingale property of certain local martingales. \textit{Probability Theory and Related Fields} 152,1-30, 2012.
\bibitem{AN} A. Neuenkirch and H. Zähle, Asymptotic error distribution of the Euler method for SDEs with non-Lipschitz coefficients. \textit{Monte Carlo Methods and Applications}, 15(4), 333-351, 2009.
\bibitem{PP1997}P. Protter and D. Talay, The Euler scheme for Lévy driven stochastic differential equations, \textit{The Annals of Probability }25.1: 393-423, 1997.
\bibitem{PP1990}P. Protter, \textit{Stochastic integration and differential equations, Version 2.1, Second Edition}, Springer, Heidelberg, 2005.




\bibitem{RD2013} D. Revuz and M. Yor, \textit{Continuous martingales and Brownian motion}.,Vol. 293. Springer, 2013.
\bibitem{TD1983} D. Talay, R\'esolution trajectorielle et analyse num\'erique des \'equations diff\'erentielles stochastiques, \emph{Stochastics}: \textit{An International Journal of Probability and Stochastic Processes} 9.4 : 275-306, 1983.
\bibitem{TD1984} D. Talay, Efficient numerical schemes for the approximation of expectations of functionals of the solution of a SDE, and applications. \emph{Filtering and control of random processes}, Springer Berlin Heidelberg, 294-313, 1984.
\bibitem{TD1990}  D. Talay and L. Tubaro, Expansion of the global error for numerical schemes solving stochastic differential equations.\textit{ Stochastic analysis and applications} 8.4 : 483-509, 1990.
\bibitem{TD1990-2} D. Talay, Simulation and numerical analysis of stochastic differential systems: a review. (Doctoral dissertation, INRIA), 1990.

\bibitem{VDV1996}  A. W. van der Vart and J. A. Wellner, \textit{Weak convergence and empirical processes: with applications to statistics}. Springer, 2013.
\bibitem{YL2002} L. Yan, The Euler scheme with irregular coefficients. \textit{The Annals of Probability} 30.3: 1172-1194, 2002.
\bibitem{YL2005} L. Yan,  Asymptotic error for the Milstein scheme for SDEs driven by continuous semimartingales.  \textit{The Annals of Applied Probability} 15, no. 4: 2706-2738, 2005.

\end{thebibliography}
\end{document}